\documentclass[a4paper,11pt,twoside,reqno,final]{amsart}
\usepackage{amsfonts}
\usepackage{amsmath}
\usepackage{mathrsfs, enumerate, amssymb}
\usepackage{tikz}\usepackage{fancybox}
\usepackage{pstricks}
\usepackage{pst-node}
\usepackage[headinclude,DIV13]{typearea}
\usepackage{hyperref}
\usepackage{color}
\usepackage{geometry}
\usepackage{graphicx, psfrag}
\usepackage{ulem}

\newtheorem{proposition}{Proposition}
\newtheorem{lemma}{Lemma}

\newtheorem{theorem}{Theorem}

\newcommand{\p}{\Bbb{P}}
\newcommand{\e}{\Bbb{E}}
\newcommand{\ud}{\mathrm{d}}

\newcommand{\R}{\mathbb{R}}

\newcommand{\eps}{\varepsilon}
\renewcommand{\P}{\mathbb{P}}
\newcommand{\N}{\mathbb{N}}
\newcommand{\E}{\mathbb{E}}
\newcommand{\Ind}[1]{\mathbf{1}_{\{#1\}}}
\newcommand{\Expo}[1]{\exp\left\{#1\right\}}
\newcommand{\Prob}[1]{\mathbb{P}\left(#1\right)}
\newcommand{\Exp}[1]{\mathbb{E}\left[#1\right]}
\newcommand{\Expx}[2]{\mathbb{E}_{#1}\left[ #2 \right]}
\newcommand{\q}{^{(q)}}

\begin{document}

\title[Asymptotic behaviour of exponential functionals of L\'evy processes]
{Asymptotic behaviour of exponential functionals of L\'evy processes with applications to random processes in random environment}
\thanks{CS thanks the Centro de Investigaci\'on en Matem\'aticas, where part of this work was done. 
This work  was partially funded by the Chair "Mod\'elisation Math\'ematique et Biodiversit\'e" of VEOLIA-Ecole Polytechnique-MNHN-F.X. and 
by the franco-mexican project PICS (CNRS) "Structures Markoviennes Auto-Similaires". SP and JCP  acknowledge support from  the Royal Society and  SP also acknowledge support from CONACyT-MEXICO Grant 351643.}
\author{Sandra Palau}
\address{Centro de Investigaci\'on en Matem\'aticas A.C. Calle Jalisco s/n. 36240 Guanajuato, M\'exico}
\email{sandra.palau@cimat.mx}
\author{Juan Carlos Pardo}
\address{Centro de Investigaci\'on en Matem\'aticas A.C. Calle Jalisco s/n. 36240 Guanajuato, M\'exico}
\email{jcpardo@cimat.mx}
\author{Charline Smadi}
\address{Irstea, UR LISC, Laboratoire d'Ing\'enierie des Syst\`emes
Complexes, 9 avenue Blaise Pascal-CS 20085, 63178 Aubi\`ere, France and
Department of Statistics, University of Oxford, 1 South Parks Road, Oxford
OX1 3TG, UK}
\email{charline.smadi@polytechnique.edu}

\date{}

\maketitle
\vspace{0.2in}

\begin{abstract} 
\noindent Let $\xi=(\xi_t, t\ge 0)$ be a real-valued L\'evy process and define  its associated exponential functional as follows
\[
I_t(\xi):=\int_0^t \exp\{-\xi_s\}{\ud s}, \qquad t\ge 0.
\]
Motivated by applications to stochastic processes in random environment,  we study the asymptotic behaviour of 
\[
\mathbb{E}\Big[F\big(I_t(\xi)\big)\Big] \qquad \textrm{as}\qquad t\to \infty,
\]
where  $F=(F(x),x\geq 0)$ is a function with polynomial decay at infinity and which is non increasing for large $x$. 
In particular, under some exponential moment conditions on $\xi$, we find five different regimes that depend  on the shape of the Laplace exponent of $\xi$.  
Our proof relies on a discretization of the exponential functional $I_t(\xi)$ and is closely related to the behaviour
of functionals of semi-direct products of random variables.

We apply our results to three questions associated to  stochastic processes in random environment. We first consider 
the asymptotic behaviour of extinction and explosion for self-similar continuous state branching processes in a 
L\'evy random environment. Secondly, we focus on the asymptotic behaviour of the mean population size in a  
model with competition or  logistic growth which is affected by a 
L\'evy random environment and finally, we study the tail behaviour of the maximum of a diffusion in a L\'evy random environment. \\


\end{abstract}

\section{Introduction and main results}\label{section1}

A one-dimensional L\'evy process is a stochastic process issued from the origin with stationary and independent increments and almost sure c\`adl\`ag paths. We write $\xi=(\xi_t : t\geq 0)$ for its trajectory and $\mathbb{P}$ for its law. 
The process $\xi$ is a strong Markov process, and for each $x\in\mathbb{R}$, we denote by 
$\mathbb{P}_x$ its  law when issued from $x$ with the understanding that $\mathbb{P}_0 = \mathbb{P}$. 
The law of a L\'evy process is characterized by its one-time transition probabilities. 
In particular there always exists a triple $(\mu, \rho, \Pi)$ where $\mu\in\mathbb{R}$, 
$\rho\in \mathbb{R}$ and $\Pi$ is a measure on $\mathbb{R}\backslash\{0\}$ satisfying the integrability condition $\int_{\mathbb{R}}(1\wedge x^2)\Pi({\rm d}x)<\infty$, such that, for all $z \in\mathbb{R}$ 
\begin{equation*}
\mathbb{E}[e^{i z \xi_t}] = e^{ t \psi( i z)},
\end{equation*}
where the Laplace exponent $\psi(z)$ is given by the L\'evy-Khintchine formula
\begin{equation*}
\psi(z) = \frac{1}{2}\rho^2 z^2 +  \mu z + \int_{\mathbb{R}}\left( e^{ z x} -1- z x \ell(x)\right)\Pi({\rm d}x), \qquad z\in \R.
\end{equation*}
Here, $\ell(x)$ is the cutoff function which is usually taken to be $\ell(x)= \mathbf{1}_{\{|x|<1\}}$.  Whenever the process $\xi$ has finite mean,  we will take $\ell(x)\equiv 1$. 

In this paper, we are interested in studying the exponential functional of $\xi$, {defined by}
\[
I_t(\xi):=\int_0^t e^{-\xi_s}{\ud s}, \qquad t\ge 0.
\]
In recent years there has been a general recognition that exponential functionals of L\'evy
processes play an important role in various domains of
probability theory such as self-similar Markov processes,  
generalized Ornstein-Uhlenbeck processes, random processes in random environment, 
fragmentation processes, branching processes, mathematical finance, Brownian
motion on hyperbolic spaces, insurance
risk, queueing theory, to name but a few {(see \cite{MR2178044,carmona,kypa} and references therein)}. 

There is a vast literature about  exponential functionals of  L\'evy processes drifting to $+\infty$ or killed at an independent 
exponential time $\mathbf{e}_q$ with parameter $q\ge 0$, see for instance \cite{ BLM, MR2178044, PRV}.  For a L\'evy process $\xi$ satisfying 
one of these assumptions, $I_\infty(\xi)$ or $I_{\mathbf{e}_q}(\xi)$ is 
finite almost surely with an absolute continuous density. Most of the known results on $I_\infty(\xi)$ and $I_{\mathbf{e}_q}(\xi)$ are 
related to the knowledge of their densities or the behaviour of their tail distributions. According to Theorem 3.9 in Bertoin et al. \cite{BLM}, 
there exists a density for  $I_\infty(\xi)$, here denoted by $h$. In the case when 
$q>0$, the existence of the density of $I_{\mathbf{e}_q}(\xi)$ appears  in Pardo et al. \cite{PRV}.  
Moreover,  according to Theorem 2.2. in Kuznetsov et al. \cite{KPS},  under the assumption that 
$\mathbb{E}[|\xi_1|]<\infty$, the density $h$  is completely determined by the following integral equation: for $v>0$,
\begin{equation}\label{expfunctdensity}
\begin{split}
\mu\int_v^\infty h(x)\ud x +\frac{\rho^2}{2}vh(v)&+\int_v^\infty \overline{\overline{\Pi}}^{(-)}\left(\ln \frac{x}{v}\right) h(x)\ud x \\
&+\int_0^v \overline{\overline{\Pi}}^{(+)}\left(\ln \frac{x}{v}\right) h(x)\ud x
+\int_v^\infty\frac{h(x)}{x}\ud x =0,
\end{split}
\end{equation}
where 
\[
 \overline{\overline{\Pi}}^{(+)}(x)=\int_x^\infty\int_y^\infty \Pi(\ud z)\ud y\qquad \textrm{and}\qquad  \overline{\overline{\Pi}}^{(-)}(x)=\int_x^\infty\int_{-\infty}^{-y} \Pi(\ud z)\ud y.
\]
We refer to  \cite{BLM, MR2178044, KPS, PRV}, and the references therein, for more details about these facts.\\

 In this paper, we are interested in the case when the L\'evy process $\xi$ does not satisfy such conditions, 
 in other words when $I_t(\xi)$ does not converge almost surely to a finite random variable, as $t$ goes to $\infty$. More precisely, 
 one of our aims is to study the asymptotic behaviour of 
\[
\mathbb{E}\Big[F\big(I_t(\xi)\big)\Big] \qquad \textrm{as}\qquad t\to \infty,
\]
where  $F$ is a function which, at infinity, is non-increasing and has polynomial decay, under some exponential moment conditions 
on $\xi$. We find five regimes that depend on the shape of the Laplace exponent $\psi$.
These results
will be applied in Section \ref{sectionappli} for some 
particular cases such as
\[
F(x)=x^{-p}, \qquad F(x)=1-e^{x^{-p}},\quad F(x)=e^{-x},\quad \mbox{ or }\quad  F(x)=\frac{a}{b+x} \quad \textrm{for } \quad a,b,p,x> 0 .
\]
Up to our knowledge, the case when the exponential functional of a L\'evy process does not converge has only been studied in a few papers and not in its most general form, 
see for instance \cite{BPS,BH,PP}. 
In all these papers, the main motivation comes from the study of some asymptotic properties of  stochastic processes in random environment, that we briefly describe below. \\

A recent manuscript authored by  Li and Xu  \cite{li2016asymptotic}, that appeared on Arxiv at the same time  this manuscript was prepared, also studied 
the asymptotic behaviour of exponential functionals for L\'evy processes. They obtained interesting results which are similar as the one we present below using 
fluctuation theory for L\'evy processes and the knowledge of L\'evy  processes  conditioned  to stay positive. Our approach is completely different and is 
based on a discretization of the exponential functional of L\'evy processes and on  the asymptotic behaviour of functionals of semi-direct products of random 
variables which was described by Guivarc'h and Liu \cite{GL}.

We would like to stress that Bansaye et al. \cite{BPS} already used a discretization technique to get the asymptotic behaviour of 
the survival probability for branchnig processes with catastrophes. However the  discretization used in \cite{BPS} is different and only covers the case of 
compound Poisson processes.

Branching processes in random environment (BPREs) have been introduced and first studied by 
Smith and Wilkinson. This type of process has attracted considerable interest in the last
decade (see for instance  \cite{afanasyev2012limit,Afanasyev2005,babe,bo10} and the references therein). 
One of the reason is that BPREs are more realistic models than classical branching processes. And, from the mathematical point of view, 
they have interesting features such as a phase transition in the subcritical regime. Scaling limits for BPREs have been studied by 
Kurtz \cite{Kurtz} in the  continuous case and more recently  by Bansaye and Simatos \cite{bansima} in a more general setting.

Continuous state branching processes (CB-processes for short) in random environment, the continuous analogue in time and state space of BPREs, can be  defined as a strong solution of a particular 
stochastic differential equation. They have been studied recently by several authors in different settings. Boeinghoff and Hutzenthaler 
\cite{BH} and Bansaye et al. \cite{BPS} studied  extinction rates for branching diffusions in a Brownian environment and 
branching processes in a random environment driven by a L\'evy process with bounded variations, respectively.
Motivated by these works, Palau and Pardo \cite{PP} studied the long-term behaviour (extinction, explosion, conditioned version) of branching processes in a Brownian random environment. 
In all these manuscripts, the existence of such processes is proved via a stochastic differential equation and  it is observed that  the speed of extinction  is related to an exponential functional of a L\'evy process 
which is associated to the random environment. Moreover, and similarly to  the case of BPREs,  a phase transition in the subcritical regime appears. 
Recently, branching processes in a  general L\'evy random environment were  introduced by Palau and Pardo in \cite{PP1}.\\

Exponential functionals occur naturally in the study of some models of diffusions in \label{defdiff}
random environment, which we now describe informally. Associated with a stochastic process $V = (V(x), x\in \mathbb{R})$ such that $V(0) = 0$, a diffusion $X_V = (X_V(t), t\ge 0)$ in the random potential $V$ is, loosely speaking, a solution to the stochastic differential equation
\[
\ud X_V(t) = \ud \beta_t -\frac{1}{2} V^\prime(X_V (t))\ud t, \qquad  X_V(0) = 0 , 
\]
where $(\beta_t, t\ge  0)$ is a standard Brownian motion independent of $V$. More rigorously, the process $X_V$ should be considered as a diffusion 
whose conditional generator, given $V$, is:
\[
\frac{1}{2}\exp(V(x))\frac{\ud}{\ud x}\left(e^{-V(x)}\frac{\ud}{\ud x}\right).
\]
Observe that from Feller's construction of such diffusions, the potential $V$ does not need to be differentiable.
 Kawazu and Tanaka \cite{kawazu1993maximum} studied the asymptotic behaviour of the tail of the distribution of the maximum 
 of a diffusion in a drifted Brownian potential. Carmona et al. \cite{carmona} considered the case when the potential is a 
 L\'evy process whose jump structure  is of bounded variation. 
 More precisely, they studied the following  question: How fast does $\P(\max_{t\geq 0}X_V(t)>x)$ decay as $x$ go to infinity?
  From these works, we know that
  $$\Prob{\max_{t\ge 0}X_V(t)>x}=\Exp{\frac{\tilde{I}}{\tilde{I}+I_x(-V)}}$$
  where 
  $$\tilde{I}=\int_{-\infty}^{0}e^{V(t)}\ud t \quad \mbox{ and }\quad I_x(-V)=\int_0^x e^{V(t)}\ud t$$
  are independent. As a consequence, exponential functionals play an essential role in this domain.
 \\

Let us now state our main results.  Assume that
\begin{equation}\label{tetam}
\theta^+=\sup\left\{\lambda>0: \psi(\lambda)<\infty\right\}
\end{equation}
exists and is positive. In other words, the Laplace exponent of the L\'evy process $\xi$ can be defined  on $[0, \theta^+)$, see for instance Lemma 26.4 
in Sato \cite{Sa}.
\label{discuLaplace}
Besides, $\psi$ satisfies
$$\psi(\lambda)=\log \Exp{e^{\lambda \xi_1}}, \qquad \lambda\in[0,\theta^+).$$
From Theorem 25.3 in \cite{Sa}, $\psi(\lambda)<\infty$ is equivalent to 
\begin{equation}\label{finite psi}
\int_{\{|x|>1\}} e^{\lambda x}\,\Pi(\ud x)<\infty.
\end{equation}
Moreover $\psi$ belongs to $C^\infty([0,\theta^+))$ with $\psi(0)=0$, $\psi^\prime(0+)\in [-\infty, \infty)$ 
and $\psi^{\prime\prime}(\lambda)>0,$ for  $\lambda\in (0, \theta^+)$ (see  Lemma 26.4 in \cite{Sa}). 
{Hence}, the Laplace exponent $\psi$ is a convex function on  $[0, \theta^+)$ implying that either 
it is positive or it may have another root on $(0, \theta^+)$.
In the latter scenario,   $\psi$ has at most one  global minimum on $(0, \theta^+)$.
 Whenever such a global minimum exists, we denote by $\tau$ the position where it is reached. 
 As we will see  below, this parameter is relevant to determine the asymptotic
 behaviour of $\Exp{I_t(\xi)^{-p}}$, for $0< p<\theta^+$.

Let us introduce the  exponential change of measure known as the Esscher transform. 
According to Theorem 3.9 in Kyprianou \cite{MR2250061}, for 
any $\lambda$ such that \eqref{finite psi} is satisfied, we can perform the following change of measure
\begin{equation}\label{escheerk}
 \frac{\ud \P^{(\lambda)}}{\ud \P}\bigg|_{\mathcal{F}_t}= e^{\lambda \xi_t-\psi(\lambda)t}, 
 \qquad t\ge 0
 \end{equation}
where $(\mathcal{F}_t)_{t\ge 0}$ is the natural filtration generated by $\xi$ which is naturally completed.  Moreover, 
under $\P^{({\lambda})}$ the process $\xi$ is still a L\'evy process with Laplace exponent given by 
\[
\psi_{{\lambda}}(z)=\psi(\lambda+z)-\psi(\lambda),\qquad z\in\R.
\]

\begin{theorem}\label{polinomial}
Assume that $0<p<\theta^+$.
\begin{enumerate}
\item[i)] If $\psi'(0+)>0$, then
$$ \lim_{t \to \infty} \Exp{I_t(\xi)^{-p}} = \Exp{I_{\infty}(\xi)^{-p}}{>0}.$$
\item[ii)] If $\psi'(0+)=0$ and $\psi^{\prime\prime}(0+)<\infty$, then there exists a positive constant $c_1$ such that
$$ \lim_{t \to \infty} \sqrt{t}  \Exp{I_t(\xi)^{-p}} = c_1.$$
\item[iii)] Assume that  $\psi'(0+)<0$ 
\begin{enumerate}
 \item[a)] if $\psi'(p)<0$, then 
$$ \lim_{t \to \infty} e^{-t\psi(p)}   \Exp{I_t(\xi)^{-p}}=\mathbb{E}^{(p)}[I_{\infty}(-\xi)^{-p}]{>0}.$$
\item[b)] if $\psi'(p)=0$ and  $\psi^{\prime\prime}(p)<\infty$ ,
then there exists a positive constant $c_2$ such that
$$\lim_{t \to \infty}\sqrt{t} e^{-t\psi(p)} \Exp{I_t(\xi)^{-p}} = c_2 .$$
\item[c)]   $\psi'(p)>0$ and $\psi^{\prime\prime}(\tau)<\infty$
 then
$$\Exp{I_t(\xi)^{-p}}=o(t^{-1/2}e^{t\psi(\tau)}), \qquad as \quad t\rightarrow \infty.$$
Moreover if we {also} assume that $\xi$ is non-arithmetic (or non-lattice) then
$$\Exp{I_t(\xi)^{-p}}=O(t^{-3/2}e^{t\psi(\tau)}), \qquad as \quad t\rightarrow \infty.$$
\end{enumerate}
\end{enumerate}
\end{theorem}
It is important to note that for  any $q>0$ satisfying (\ref{finite psi}), we necessarily have  that $\E \left[ I_t(\xi)^{-q} \right] $ is finite 
for all $t>0$.
Indeed, since $(e^{q \xi_t-t\psi(q)}, t\ge 0)$ is a positive martingale, we deduce from  $L_1$-Doob's 
inequality (see for instance \cite{acciaio2012trajectorial}) and the Esscher transform \eqref{escheerk}, that the following series of inequalities hold: for $t \leq 1$,
\begin{equation}\label{mdoob}
\begin{split}
\Exp{I_t(\xi)^{-q}} &\le t^{-q} \mathbb{E}\left[\sup_{0\le u\le 1} e^{q \xi_u}\right]\le  t^{-q} e^{\psi(q)\lor 0}\mathbb{E}\left[\sup_{0\le u\le 1} e^{q \xi_u-u\psi(q)}\right]\\
&\le  t^{-q} \frac{e^{1+\psi(q)\lor 0}}{e-1}\left(1+\mathbb{E}^{(q)}\Big[ q\xi_1-\psi(q)\Big]\right)=
t^{-q} \frac{e^{1+\psi(q)\lor 0}}{e-1} [1+q \psi'(q)-\psi(q)], 
\end{split}
\end{equation} 
which is finite. The finiteness for $t >1$ follows  from the fact that $I_t(\xi)$ is non-decreasing.

We are now interested in extending the above result for a class of functions which have
polynomial decay and are non-increasing at $\infty$. As we will see below such extension is not straightforward and need more conditions on the exponential 
moments of  the L\'evy process $\xi$.

 For simplicity, we write
$$\mathcal{E}_F( t):=\E\left[F(I_t(\xi))\right],$$
where $F$ belongs to a particular class of continuous functions on $\R_+$ that we will introduce below.  We 
 assume that  the Laplace exponent $\psi$ of $\xi$  is well defined  on the interval $ (\theta^-, \theta^+)$, where 
$$ \theta^-:= \inf \{\lambda <0: \psi(\lambda)<\infty\},$$
and $\theta^+$ is defined as in (\ref{tetam}).  Recall that $\psi$ is a convex function that belongs to $C^\infty((\theta^-, \theta^+))$ with $\psi(0)=0$,  $\psi^\prime(0+)\in [-\infty, \infty)$ 
and $\psi^{\prime\prime}(\lambda)>0,$ for  $\lambda\in(\theta^-, \theta^+)$. Also recall that $\tau\in[0,\theta^+)$ is the position where the minimum of $\psi$ is reached. 

Let  $\mathbf{k}$ be 
a positive constant. We will consider functions $F$ satisfying one of the following conditions:  
There exists $x_0\geq 0$ such that $F(x)$ is non-increasing for $x \geq x_0$, and
\begin{itemize}
\item[({\bf A1})] $F$ satisfies
\begin{equation*} 
 F(x)=\mathbf{k}(x+1)^{-p}\Big[ 1+(1+x)^{-\varsigma}h(x) \Big] ,\qquad \mbox{ for all } x>0,
\end{equation*}
where $0<p \le\tau$, $\varsigma\geq 1$ and  $h$ is a Lipschitz function which is bounded.

\item [({\bf A2})] $F$ is an  H\"older  function with index  $\alpha>0$ satisfying
 \begin{equation*}
 F(x)\leq \mathbf{k} (x+1)^{-p},\qquad \mbox{ for all } x>0,
\end{equation*}
with $ p >\tau.$
\end{itemize}

\begin{theorem}\label{funcLev}Assume that $0<p<\theta^+$.
We have the following five regimes  for the asymptotic behaviour of $\mathcal{E}_F(t)$ for large $t$.
 \begin{enumerate}
 \item[i)] If $\psi'(0+)>0$ and $F$ is a positive and continuous function which is bounded, then 
 $$ \lim_{t \to \infty} \mathcal{E}_F(t)= \mathcal{E}_F(\infty).$$
 \item[ii)] If $\psi'(0+)=0$, 
$F$ satisfies ({\bf A2}) and $\theta^-<0$, 
then 
 there exists a positive constant $c_3$ such that 
 $$ \lim_{t \to \infty} \sqrt{t}   \mathcal{E}_F(t)=c_3.$$
 \item[iii)] Suppose that $\psi'(0+)<0$: 
\begin{enumerate}
 \item[a)]   If $F$ satisfies ({\bf A1}) and $\psi^\prime(p)<0$,
then, 
$$\lim_{t \to \infty} e^{-t\psi(p)}   \mathcal{E}_F(t) = 
{\lim_{t \to \infty} e^{-t\psi(p)}   \mathbf{k}\E\left[ I_t(\xi)^{-p} \right] =} \mathbf{k}\E^{(p)}\left[ I_\infty(-\xi)^{-p} \right].$$
 \item[b)] If $F$ satisfies  ({\bf A1}),  $\psi^\prime(p)=0$ and $\psi^{\prime\prime}(p)<\infty$, 
 then, 
 $$
 \lim_{t \to \infty}\sqrt{t}e^{-t\psi(p)} \mathcal{E}_F(t) = 
 {\lim_{t \to \infty} \sqrt{t}e^{-t\psi(p)}  \mathbf{k}\E\left[ I_t(\xi)^{-p} \right]=} \mathbf{k} c_2 ,
$$
{where $c_2$ has been defined in point iii) b) of Theorem \ref{polinomial}.}
\item[c)] If $F$ satisfies ({\bf A2}), $\psi^\prime(p)>0$ and  $\tau+p<\theta^+$
then there exists a positive constant $c_4$ such that
$$\lim_{t \to \infty} t^{3/2} e^{-t\psi(\tau)} \mathcal{E}_F(t) = c_4. $$
\end{enumerate}
\end{enumerate}
\end{theorem}
The remainder of the paper is structured as follows. In Section 2, we  apply our results to self-similar 
CB-processes, a population model with competition and diffusions whose dynamics are perturbed by a L\'evy random 
environment. In particular, we study the asymptotic behaviour of the probability of extinction and 
explosion  for some classes of self-similar CB-processes in a L\'evy random environment. 
For the population model with competition, we describe the asymptotic behaviour of its mean. For the diffusion in a
L\'evy random environment, we  provide the asymptotic behaviour of the tail probability of its global maximum. Finally, Section 3 is devoted 
to the proofs of Theorems \ref{polinomial} and \ref{funcLev}.

\section{Applications} \label{sectionappli}

\subsection{Self-similar CB-processes in a L\'evy random environment.}

A $[0,\infty]$-valued strong Markov process $Y =(Y_t, t\geq 0)$  with probabilities $(\mathbb{P}_x, \, x\geq 0)$ is called a 
CB-process if it has c\`adl\`ag paths  and its law satisfies the branching
property;  i.e. for any $x,y\geq 0$, $\mathbb{P}_{x+y}$ is equal in law to the convolution of $\mathbb{P}_x$ and $\mathbb{P}_y$.  The law of $Y$ is completely characterized by its Laplace transform
\begin{equation*}
  \Expx{x}{e^{-\lambda Y_t}}=e^{-xu_t(\lambda)},\qquad \forall x>0,\ t\geq 0,
\end{equation*}
where $u$ is a differentiable  function in $t$ satisfying
\begin{equation*}
\frac{\partial u_{t}(\lambda)}{\partial t}=-\Psi(u_{t}(\lambda)), \qquad u_{0}(\lambda)=\lambda.
\end{equation*}
The function $\Psi$ is known as the branching mechanism of $Y$ and satisfies the celebrated 
L\'evy-Khintchine formula
\begin{equation*}\label{lk}
\Psi(\lambda)=-a\lambda+\gamma^2
\lambda^2+\int_{(0,\infty)}\big(e^{-\lambda x}-1+\lambda x{\mathbf 1}_{\{x<1\}}\big)\mu(\ud x),\qquad \lambda\geq 0,
\end{equation*}
where $a\in \mathbb{R}$, $\gamma\geq 0$ and $\mu$ is a measure concentrated on $(0,\infty)$ such that $\int_{(0,\infty)}\big(1\land x^2\big)\mu(\ud x)<\infty.$

Here we are interested in the case where the branching mechanism is stable, that is to say $$\Psi(\lambda)= c_\beta\lambda^{\beta+1}, \qquad \lambda \in \R,$$
for some $\beta\in(-1,0)\cup(0,1]$ and $c_\beta$ is such that $\beta c_{\beta}>0$.  We call its associated CB-process a self-similar CB-process. 
Under this assumption, the process $Y$ can also be defined as the unique non-negative strong solution of the following SDE 
(see for instance \cite{FuLi})
\begin{align*}
 Y_t=&Y_0+\Ind{\beta=1}\int_0^t \sqrt{2c_\beta Y_s}\ud B_s +\Ind{\beta\neq 1}\int_0^t\int_0^{\infty}\int_0^{Y_{s-}}z\widehat{N}(\ud s,\ud z,\ud u), 
\end{align*}
where $B=(B_t,t\geq 0)$ is a standard Brownian motion, 
$N$ is a Poisson random measure independent of $B$ with intensity 
\begin{align*}
\frac{c_\beta\beta(\beta+1)}{\Gamma(1-\beta)}\frac{1}{z^{2+\beta}}\ud s \ud z\ud u,
\end{align*}
 $\widetilde{N}$ is its  compensated version and
 \[
 \widehat{N}(\ud s,\ud z,\ud u) =\left\{ \begin{array}{ll}
 N(\ud s,\ud z,\ud u)  &\textrm{ if $\beta \in(-1,0),$}\\
 \widetilde{N} (\ud s,\ud z,\ud u) & \textrm{ if $\beta \in (0,1).$}
 \end{array} \right. 
 \]

According to Palau and Pardo \cite{PP1}, we can define a self-similar branching process whose dynamics are affected by a L\'evy random environment (SSBLRE) as the unique non-negative strong solution of the stochastic differential equation 
\begin{align}\label{csbplre}
 Z_t=&Z_0+\Ind{\beta=1}\int_0^t \sqrt{2c_\beta Z_s}\ud B_s +\Ind{\beta\neq 1}\int_0^t\int_0^{\infty}\int_0^{Z_{s-}}z\widehat{N}(\ud s,\ud z,\ud u)+\int_0^t  Z_{s-}\ud S_s,
\end{align}
where
\begin{equation}\label{Lenvi}
 S_t=\alpha t+\sigma W_t+\int_0^t\int_{(-1,1)}(e^v-1) \widetilde{N}^{(e)}(\ud s,\ud v)+\int_0^t\int_{\R\setminus (-1,1)}(e^v-1){N}^{(e)}(\ud s,\ud v),
 \end{equation} 
 $\alpha\in \mathbb{R}$, $\sigma\geq 0, W=(W_t,\ t\geq 0)$ is a standard Brownian motion, ${N}^{(e)}$ is a Poisson random measure in $\R_+\times \R$ independent 
 of $W$ with intensity $\ud s\pi(\ud y)$, $\widetilde{N}^{(e)}$ represents its compensated version and $\pi$ is a $\sigma$-finite measure concentrated on 
 $\R\setminus \{0\}$ such that
 $$\int_{\R} (1\wedge v^2)\pi(\ud v)<\infty.$$ 
 In addition, $Z$ satisfies the strong Markov property and, conditioned on the environment, the branching property. 
In what follows, we introduce the auxiliary L\'evy process
\begin{equation}\label{auxiliary}
 K_t=\mathbf{d} t+\sigma W_t+\int_0^t\int_{(-1,1)}v\widetilde{N}^{(e)}(\ud s,\ud v)+\int_0^t\int_{\R\setminus (-1,1)}v N^{(e)}(\ud s,\ud v),
 \end{equation}
where 
\[
\mathbf{d}=\alpha-\frac{\sigma^2}{2}-\int_{(-1,1)}(e^v-1-v)\pi(\ud v).
\]

The Laplace transform {of }$Z_t e^{-K_t}$ can be computed explicitly and provides a closed formula for the probabilities of survival and 
non explosion of $Z$.
  {}Indeed according to Proposition 1 in \cite{PP1}, 
if $(Z_t, t\geq 0)$ is a stable SSBLRE with index $\beta\in(-1, 0)\cup (0,1]$, then for all $z, \lambda>0$ and $t\geq 0$,  
\begin{align}\label{laplacecondicionadol}
\e_z\Big[\exp\Big\{-\lambda Z_t e^{-K_t}\Big\}\Big| K\Big]=\Expo{-z\left(\lambda ^{-\beta}+\beta c_\beta\int_0^t e^{-\beta K_u} \ud u\right)^{-1/\beta}}.
\end{align}
Here, we are interested  in two events which are of immediate concern for the  process $Z$, {\it explosion} and {\it extinction}. 
The event of {\it explosion} at fixed time $t$, is given by  $\{Z_t=\infty\}$, and the event 
$\{\exists \,t >0, Z_t=0 \}$ is referred as {\it extinction}.

\subsubsection{Speed of explosion of SSBLRE}

Let us first study the event of {\it explosion} for self-similar CB-processes in a L\'evy random environment. 
It is important to note that this event has only been studied  
in the case when the random  environment  is driven by a Brownian motion with drift, see \cite{PP}.
From {Equation \eqref{laplacecondicionadol}} and letting $\lambda$ goes to 0, we deduce
\begin{align}\label{eqexplo}
\mathbb{P}_z\Big( Z_t < \infty\Big|K\Big)=
\mathbf{1}_{\{\beta >0\}}+
\mathbf{1}_{\{\beta <0\}}\exp\left\{-z\left(\beta c_\beta\int_0^t e^{-\beta K_u} \ud u\right)^{-1/\beta}\right\}\qquad \textrm{a.s.}
\end{align}
Let us focus on the most interesting case, $\beta \in (-1,0)$. 
{Recall that when there is no environment, a self-similar CB-process explodes at time $t$ with probability  $1-\exp\left\{-z(\beta c_\beta t)^{-1/\beta}\right\}$. Under the presence of a L\'evy random environment  three different regimes appear for the asymptotic behaviour of the non-explosion probability}. We call these regimes {\it subcritical-explosion}, {\it critical-explosion} or 
{\it supercritical-explosion} depending on whether this probability stays positive, converges to zero polynomially fast or converges to 
zero exponentially fast.\\
Before stating our result, let us introduce the Laplace transform of the L\'evy process $K$ by 
\begin{equation}\label{defk}
\kappa(\theta)=\log \E[e^{\theta K_1}], 
\end{equation}
whenever it exists (see discussion on page \pageref{discuLaplace}).

 We assume that  the Laplace exponent $\kappa$ of $K$  is well defined  on the interval $ (\theta_K^-, \theta_K^+)$, where 
\begin{equation*} \theta_K^-:= \inf \{\lambda <0: \kappa(\lambda)<\infty\} \quad \text{and} \quad 
\theta_K^+:= \sup \{\lambda >0: \kappa(\lambda)<\infty\}.\end{equation*}
If $\kappa$ reaches a global minimum on $(0,\theta^+)$, we denote by $\tau$ for its position. As we  see below, the asymptotic behaviour of 
the probability of explosion depends on the sign of 
\[
\mathbf{m}=\kappa^\prime(0+).
\]
\begin{proposition}\label{proexplosion}
Let $(Z_{t}, t\geq 0)$ be the SSBLRE with index $\beta\in(-1,0)$ defined by the SDE 
(\ref{csbplre}) with $Z_0=z>0$, and recall the definition of the random environment $K$ in \eqref{auxiliary}.  Assume that $0<\theta_K^{+}$.
\begin{itemize}
\item[i)] Subcritical-explosion. If $\mathbf{m}>0$, then, for every $z>0$
\begin{equation*}
\underset{t\rightarrow\infty}{\lim}\mathbb{P}_z\Big( Z_t < \infty\Big)=\mathbb{E}\left[ \exp\left\{-z\left(\beta c_\beta\int_0^\infty e^{-\beta K_u} \ud u\right)^{-1/\beta}\right\}\right]
>0. 
\end{equation*}
\item[ii)]  Critical-explosion. If $\mathbf{m}=0$ and $\theta_K^{-}<0$,
then
for every $z>0$ there exists $c_1(z)>0$ such that
\begin{equation*}
\underset{t\rightarrow\infty}{\lim}\sqrt{t} \mathbb{P}_z\Big( Z_t < \infty\Big)=c_1(z).  
\end{equation*}
\item[iii)] Supercritical-explosion. If $\mathbf{m}<0$. Then
for every $z>0$ there exists $c_2(z)>0$ such that
\begin{equation*}
\underset{t\rightarrow\infty}{\lim}t^{\frac{3}{2}} e^{-t\kappa(\tau)} \mathbb{P}_z\Big( Z_t < \infty\Big)=c_2(z). 
\end{equation*}
\end{itemize}
\end{proposition}
\begin{proof}{Observe} that the function 
$$F: x \in \R_+ \mapsto \exp ( -x^{-1/\beta} ) $$
is non-increasing, continuous, bounded, and 
satisfies Assumption  (\textbf{A2}) for every positive $p$. Hence Proposition \ref{proexplosion} is a direct application of Theorem 
\ref{funcLev} points i), ii) and iii) c).
\end{proof}

\subsubsection{Speed of extinction of SSBLRE} 
Let us now focus on the {survival} probability.
Throughout this section, we assume that  $\beta\in(0,1]$. 
From {Equation \eqref{laplacecondicionadol}} and taking $\lambda$ go to $\infty$, we get
\begin{align*}
\mathbb{P}_z\Big( Z_t >0\Big|K\Big)=1-
\exp\left\{-z\left(\beta c_\beta\int_0^t e^{-\beta K_u} \ud u\right)^{-1/\beta}\right\}\qquad \textrm{a.s.}
\end{align*}
Similarly as for the explosion probability, the asymptotic behaviour of 
the {probability depends on the sign of 
$\mathbf{m}=\kappa^\prime(0+).$
But in contrast with  the explosion probability, five regimes appear, and a second parameter to take into account is the sign of 
$$\mathbf{m}_1=\kappa^\prime(1). $$

\begin{proposition} \label{proextinction}
Let $(Z_{t}, t\geq 0)$ {be a SSBLRE with index  $\beta\in(0, 1]$} defined by the SDE 
(\ref{csbplre}) with $Z_0=z>0$, and recall the definition of the random environment $K$ in \eqref{auxiliary}. 
Assume that $1<\theta^+_K$.
\begin{enumerate}
\item[i)] Supercritical case. If $\mathbf{m}>0$, then for every $z>0$
{\begin{equation*}
\underset{t\rightarrow\infty}{\lim}\mathbb{P}_z\Big( Z_t >0\Big)=
\mathbb{E}\left[1-\exp\left\{-z\left(\beta c_\beta\int_0^\infty e^{-\beta K_u} \ud u\right)^{-1/\beta}\right\}\right]>0. 
\end{equation*}}
\item[ii)] Critical case. If  $\mathbf{m}=0$
and $\theta^-<0$, 
then for every $z>0$, there exists $c_3(z)>0$ such that 
$$\underset{t\rightarrow\infty}{\lim}\sqrt{t}\P_{z}(Z_t>0) = c_3(z). $$
\item[iii)]
 Subcritical case. Assume that $\mathbf{m}<0$ , then 
\begin{enumerate}
 \item[a)] (Strongly subcritical regime). If $\mathbf{m}_1<0$,
 then there exists $c_1>0$ such that for every ${z>0}$,
$$ \underset{t\rightarrow\infty}{\lim}e^{-t\kappa(1)}\P_{z}(Z_t>0) =c_1  z , $$
 \item[b)]  (Intermediate subcritical regime) If $\mathbf{m}_1=0$, 
 then there exists $c_2>0$ such that for every $z>0$,
$$\underset{t\rightarrow\infty}{\lim}\sqrt{t}e^{-t\kappa(1)}\P_{z}(Z_t>0) = c_2 z , $$
 \item[c)]  (Weakly subcritical regime) If $\mathbf{m}_1>0$, then for every $z>0$, there exists $c_4(z)>0$ such that
$$\underset{t\rightarrow\infty}{\lim}t^{3/2}e^{-t\kappa(\tau)}\P_{z}(Z_t>0) = c_4(z).  $$
\end{enumerate}
\end{enumerate}
\end{proposition}

\begin{proof}
 This is a direct application of Theorem \ref{funcLev}, with $(\xi_t, t \geq 0) = (\beta K_t, t \geq 0)$ and 
 $F(x)=1-\exp(-z(\beta c_\beta x)^{-1/\beta})$.
\end{proof}
In the strongly and intermediate subcritical cases $a)$ and $b)$, $\E [Z_t] $  provides the  exponential decay factor of 
the survival 
probability which is given by  $\kappa(1)$, and the probability of non-extinction is proportional to the initial state $z$ of 
the population. 
In the weakly subcritical case $c)$, the survival probability decays exponentially with rate  $\kappa(\tau)$, which is 
strictly smaller than  $\kappa(1)$, and $c_4$ may not be proportional to $z$ (it is also the case for $c_1$).
We refer to \cite{MR2532207}
for a result in this vein for discrete branching processes in random environment.

More generally, the results stated above can be compared to the results which appear in the literature of discrete (time and space) 
branching processes in random environment, 
see e.g. \cite{Afanasyev2005, MR1983172, GL}.
In the continuous framework, such results have been established in \cite{BH} for the Feller diffusion case (i.e. $\beta=1$)  
in a Brownian environment, in \cite{PP} for a general CB-process in a Brownian environment, and in 
\cite{BPS} for stable CB-process ($\beta \in (0,1]$) subject to random catastrophes killing a fraction of the population.

\subsection{Population model with competition in a L\'evy random environment}
We now study an extension of the competition model {introduced} in Evans et al. \cite{EHS} and studied by Palau and Pardo 
\cite{PP1}. Following \cite{PP1}, we define a logistic process with competition in a L\'evy random environment, 
$(Z_t,t \geq 0)$, as the unique strong solution of the SDE
\begin{align*}
 Z_t=Z_0+\int_0^t Z_s({r}-kZ_s)\ud s+\int_0^t  Z_{s-}\ud S_s 
\end{align*}
where $r>0$ is the drift, $k>0$ is the competition, and the environment $S$ is given by the L\'evy process {defined} in 
(\ref{Lenvi}). Moreover, the process $Z$  satisfies the  Markov property and {we have}
\begin{equation*}
Z_t=Z_0e^{K_t}\left(1+kZ_0\displaystyle\int_0^t e^{K_s}\ud s\right)^{-1},\qquad t\ge 0, 
\end{equation*}
where $K$ is the L\'evy process defined in (\ref{auxiliary}). 

The following result studies the asymptotic behaviour of $\mathbb{E}_{z}[Z_t]$, where $\mathbb{P}_z$ denotes the law of $Z$ starting 
from $z$. {Before stating our result, let us 
recall the definition of the Laplace transform $\kappa$ of $K$ in \eqref{defk} and that $\tau$ is the position of its global minimum, $\mathbf{m}=\kappa^\prime(0)$ and $\mathbf{m}_1=\kappa^\prime(1)$.} 

\begin{proposition} \label{propsize}
Assume that $1<\theta^+_K$.
For $z>0$, we have the following five regimes for the asymptotic behaviour of $\mathbb{E}_{z}[Z_t]$.
	\begin{enumerate}
		\item[i)] If $\mathbf{m}>0$, then for every $z>0$
		$$ \lim_{t \to \infty} \mathbb{E}_{z}[Z_t]= \frac{1}{k}\Exp{\frac{1}{I_{\infty}(K)}}{>0}.$$
		\item[ii)] If $\mathbf{m}=0$, then  
$$ \mathbb{E}_{z}[Z_t]=O(t^{-1/2}).$$		
		\item[iii)] Suppose that $\mathbf{m}<0$: 
		\begin{enumerate}
			\item[a)]   If  $\mathbf{m}_1<0$,	then,
			$$\lim_{t \to \infty} e^{-t\kappa(1)}  \mathbb{E}_{z}[Z_t] = \mathbb{E}^{(1)}\left[{\frac{z}{1+zkI_{\infty}(-K)}}\right]{>0} ,$$
			where  $\E^{(1)}$ denotes the Esscher transform (\ref{escheerk}) of $K$ with $\lambda=1$.
			\item[b)] If 
			$\mathbf{m}_1=0$, then there exists a positive constant $c(z,k)$ that depends on $z$ and $k$ such that
			$$
			\lim_{t \to \infty}\sqrt{t}e^{-t\kappa(1)}  \mathbb{E}_{z}[Z_t] = c(z,k).
			$$
			\item[c)] If $\mathbf{m}_1>0$ and $\tau+1<\theta^+$ then there exists a positive constant $c_1(z,k)$ that depends on $z$ and $k$ 
			 such that
			$$\lim_{t \to \infty} t^{3/2} e^{-t\kappa(\tau)}  \mathbb{E}_{z}[Z_t] = c_1(z,k). $$
		\end{enumerate}
	\end{enumerate}
	\end{proposition}
	
\begin{proof}We first recall from Lemma II.2 in \cite{MR1771664} that the  time reversal process $(K_t-K_{(t-s)^-}, 0\le s\le t)$ 
 has the same law as $(K_s, 0\le s\le t)$. Then
\begin{equation}\label{timereversal}
\begin{split}
e^{-K_t}I_t(-K)=e^{-K_t}\int_0^t e^{K_{t-s}}{\ud s}=\int_0^{t}e^{-(K_t-K_{t-s})}{\ud s}\stackrel{(d)}{=}\int_0^t e^{-K_s}{\ud s}=I_t(K),
\end{split}
\end{equation}
and
$$ (e^{-K_t},e^{-K_t}I_t(-K))\stackrel{(d)}{=} (e^{-K_t},I_t(K)), $$
{where $\stackrel{(d)}{=}$ stands for equality in law.}
The above implies that
\begin{equation}\label{evans}
\mathbb{E}_z[Z_t]=z\mathbb{E}\left[\left(e^{-K_t}+kz e^{-K_t}\displaystyle\int_0^t e^{K_s}\ud s\right)^{-1}\right]=z\mathbb{E}\left[\left(e^{-K_t}+kz I_t(K)\right)^{-1}\right].
\end{equation}
Let us now prove part i). Assume that $\mathbf{m}>0$, 
then $K$ drifts to $\infty$ and $e^{-K_t}$ converges to 0 as $t$ goes to $\infty$. By Theorem 1 in \cite{MR2178044}, $I_t(K)$ converges 
a.s. to $I_\infty(K)$, a non-negative and finite limit as $t$ goes to $\infty$. We observe that the  result follows from identity (\ref{evans}) 
and the  Monotone Convergence Theorem.

Part ii) follows from the inequality
\[
\mathbb{E}_z[Z_t]=z\mathbb{E}\left[\left(e^{-K_t}+kz I_t(K)\right)^{-1}\right]\le \mathbb{E}\left[\left(k I_t(K)\right)^{-1}\right], 
\]
and Theorem \ref{polinomial} part (ii).

Finally, we prove part iii). Observe  by applying the Esscher transform (\ref{escheerk}) with $\lambda=1$ that
\[
\mathbb{E}_z[Z_t]=ze^{\kappa(1)t}\mathbb{E}^{(1)}\left[\left(1+kz \displaystyle\int_0^t e^{K_s}\ud s\right)^{-1}\right].
\]
Part iii)-a) follows by observing that under the probability 
measure $\P^{(1)}$, the process $K$ is a L\'evy process with  mean  $\E^{(1)}[K_1]=\kappa'(1)\in (-\infty, 0)$. We then conclude as in the proof of part i)
by showing that $ {\E^{(1)} [ (1+kzI_{t}(-K))^{-1} ],} $  converges to 
{$ \E^{(1)} [ (1+kzI_{\infty}(-K))^{-1} ],$} as $t$ increases.

Finally parts iii)-b) and c) follows {from} a direct application of Theorem  (\ref{funcLev}) parts iii)-b) and c), respectively, with the function 
$F: x\in \mathbb{R}_+\mapsto  z(1+kzx)^{-1}.$
 \end{proof}

\subsection{Diffusions in a L\'evy  random environment}
Let $(V(x),x\in\R)$ be a stochastic process defined on $\R$ such that $V(0)=0$. 
As presented in the introduction,
a diffusion process $X=(X(t),t\geq 0)$ in 
a random potential $V$ is a diffusion whose conditional generator given $V$ is
$$\frac{1}{2} e^{V(x)}\frac{\ud}{\ud x}\left(e^{-V(x)}\frac{\ud}{\ud x}\right).$$
It is well known that $X$ may be constructed from a Brownian motion through suitable changes of scale and time, see Brox \cite{Brox}.

The problem that we would like to study is the following: How fast does $\P(\max_{t\geq 0}X(t)>x)$ decay as $x$ go to infinity? 
In order to make our analysis more tractable, we consider $(\xi_t,t\geq 0)$ and $(\eta_t,t\geq 0)$ two independent L\'evy processes,
and we define

\[
V(x) = \left\{ \begin{array}{lcr}
-\xi_x & \mbox{if} & x\geq 0 \\
-\eta_{-x} & \mbox{if} & x\leq 0.
\end{array}
\right.
\]
We want to determine the asymptotic behaviour of
$$\Prob{\max_{s\ge 0}X(s)>t}=\Exp{\frac{I_{\infty}(\eta)}{I_{\infty}(\eta)+I_t(\xi)}}.$$
We assume that $\eta$ drifts to $\infty$, and recall the notations of Section \ref{section1} for the Laplace exponent 
$\psi$ of $\xi$, and for $\theta^+$, $\theta^-$ and $\tau$.

\begin{proposition}Assume that $1<\theta^+$.
	\begin{enumerate}
		\item[i)] If $\psi'(0+)>0$, then 
		$$ \lim_{t \to \infty} \Prob{\underset{s\geq 0}{\max}X(s)>t}= \Exp{\frac{I_{\infty}(\eta)}{I_{\infty}(\eta)+I_{\infty}(\xi)}}{>0}.$$
		\item[ii)] If $\psi'(0+)=0$, 
		then there exists a positive constant $C_1$ that depends on the law of $I_{\infty}(\eta)$ such that 
		$$ \lim_{t \to \infty} \sqrt{t}    \Prob{\underset{s\geq 0}{\max}X(s)>t}=C_1.$$
		\item[iii)] Suppose that $\psi'(0+)<0$: 
		\begin{enumerate}
			\item[a)]   If  $\psi^\prime(1)<0$,
			then there exists 
			a positive constant $C_2$ that depends on the law of $I_{\infty}(\eta)$ such that,
			$$\lim_{t \to \infty} e^{-t\psi(1)}   \Prob{\underset{s\geq 0}{\max}X(s)>t} = C_2.$$
			\item[b)] If 
			$\psi^\prime(1)=0$ and $\psi^{\prime\prime}(1)<\infty$,
			then there exists a positive constant $C_3$ 
			that depends on the law of $I_{\infty}(\eta)$  such that
			$$
			\lim_{t \to \infty}\sqrt{t}e^{-t\psi(1)} \Prob{\underset{s\geq 0}{\max}X(s)>t} = C_3 .
			$$
			\item[c)] If  $\psi^\prime(1)>0$, and $\tau+1<\theta^+$, then 
			$$\lim_{t \to \infty}  \Prob{\underset{s\geq 0}{\max}X(s)>t} = o(t^{-1/2} e^{-t\psi(\tau)}). $$
			Moreover, if the process $\xi$ is non-arithmetic (or non-lattice) then there exists a positive constant $C_4$ 
			that depends on the law of $I_{\infty}(\eta)$  such that
			$$\lim_{t \to \infty} t^{3/2} e^{-t\psi(\tau)} \Prob{\underset{s\geq 0}{\max}X(s)>t} = C_4. $$
		\end{enumerate}
	\end{enumerate}
	Furthermore, if there exists a positive $\eps$ such that 
		$$ \E[I_{\infty}(\eta)^{{1+\eps}}]<\infty, $$
		then 
		$$ C_i=c_i\E[I_{\infty}(\eta)], \quad i \in \{2,3\}, $$
		where $(c_i,  i \in \{2,3\})$ do not depend on the law of $I_{\infty}(\eta)$.
\end{proposition}

\begin{proof}
	Since $\eta$ and $\xi$ are independent, we have
	$$\Prob{\underset{s\geq 0}{\max}X(s)>t}=\Exp{I_{\infty}(\eta)f(I_{\infty}(\eta),t)},\qquad t>0$$
	where
	$$f(a,t)=\Exp{(a+I_t(\xi))^{-1}},\qquad a,t>0.$$
	The {result} follows from an application of Theorems \ref{polinomial} and \ref{funcLev} with the function 
	\[
	F: x\in \mathbb{R}_+\mapsto  z(a+x)^{-1}.
	\]
	We only prove case ii), {as} the others are analogous. By Theorem \ref{funcLev}  there exists $c_1(a)>0$ such that 
	$$\underset{t\rightarrow\infty}{\lim}t^{1/2}f(a,t)=c_1(a).$$
	On the other hand, by Theorem 1, there exists $c_1$ such that 
	$$\underset{t\rightarrow\infty}{\lim}t^{1/2}f(0,t)=c_1.$$
	Let us define $G_t(a)=at^{1/2}f(a,t)$, and $G_t^0(a)=at^{1/2}f(0,t)$. Observe that 
	$$G_t(a)\leq G_t^0(a), \qquad \mbox{ for all }t,a\geq 0$$
	and 
	$$\underset{t\rightarrow\infty}{\lim}\Exp{G_t^0(I_{\infty}(\eta))}=c_1\Exp{I_{\infty}(\eta)}.$$
	Then, by the Dominated Convergence Theorem (see for instance \cite{dudley2002real} problem 12 p. 145),
	$$ \lim_{t \to \infty} \sqrt{t}    \Prob{\underset{s\geq 0}{\max}X(s)>t}=\underset{t\rightarrow\infty}{\lim}
	\Exp{G_t(I_{\infty}(\eta))}=\Exp{I_{\infty}(\eta)c_1(I_{\infty}(\eta))}.$$
	We complete the proof for the existence of the limits by observing that 
	$$0<C_1=\Exp{I_\infty(\eta)c_1(I_\infty(\eta))}\leq c_1\Exp{I_\infty(\eta)}<\infty.$$
	The last part of the proof consists in justifying the form of the constants $C_2$ and $C_3$ 
	under the additional condition $ \E[I_{\infty}(\eta)^{1+\eps}]<\infty$ for a positive ${\varepsilon}$. 
		For every $0\leq \eps \leq 1$, we have
		\begin{multline*}
		\frac{I_\infty(\eta)}{I_t(\xi)} -\frac{I_\infty(\eta)}{I_\infty(\eta)+I_t(\xi)}= 
		\frac{I_\infty(\eta)}{I_t(\xi)}   \frac{I_\infty(\eta)}{I_\infty(\eta)+I_t(\xi)} 
		\leq  \frac{I_\infty(\eta)}{I_t(\xi)}   \left(\frac{I_\infty(\eta)}{I_\infty(\eta)+I_t(\xi)}\right)^\eps 
		\leq  \left(\frac{I_\infty(\eta)}{I_t(\xi)}\right)^{1+\eps} 
		\end{multline*}
		Hence
		$$ 0 \leq \E\left[\frac{I_\infty(\eta)}{I_t(\xi)} -\frac{I_\infty(\eta)}{I_\infty(\eta)+I_t(\xi)}\right]
		\leq \E[(I_\infty(\eta))^{1+\eps}]\E\left[ \frac{1}{(I_t(\xi))^{1+\eps}} \right] .$$
		But from point iii)-c) of Theorem \ref{polinomial} and Equation \eqref{limit1} in the proof of Theorem \ref{funcLev}, we know that in the cases 
		iii)-a) and iii)-b), 
		$$ \E\left[ I_t(\xi)^{-(1+\eps)} \right]=o \left( \E\left[ I_t(\xi)^{-1} \right] \right). $$ 
		This ends the proof.
\end{proof}	

\section{Proofs of Theorems \ref{polinomial} and \ref{funcLev}.}
This section is dedicated to the proofs of the main results of the paper. 

We first prove Theorem  \ref{polinomial}. The proof of part ii) is based  on the following approximation technique.  Let $(N_t\q, t\ge 0)$ be a Poisson process with intensity $q>0$, which is  independent of the L\'evy process $\xi$, and  denote by 
$(\tau_n^q)_ {n\ge 0}$ its   sequence of jump times with the convention that $\tau^q_0=0$. 
For simplicity, we also introduce for $n\ge 0$,
$$\xi^{(n)}_t=\xi_{\tau_n^q+t}-\xi_{\tau_n^q},\qquad t\ge 0.$$

\noindent For $n\ge 0$, we define the following random variables
$${S}_n\q:=\xi_{\tau_n^q}, \qquad  {M}_n\q:= \sup_{\tau_n^q\leq t<\tau_{n+1}^q}\xi_t\qquad \textrm{and} \qquad {I}_n\q:= \inf_{\tau_n^q\leq t<\tau_{n+1}^q}\xi_t.$$
Observe that $({S}_n\q, n\ge 0)$ is a random walk with step distribution given by $\xi_{\tau_1^q}$ and that  $\tau_1^q$ is an exponential r.v. with parameter $q$ which is  independent of $\xi$.

Similarly for the process $\xi^{(n)}$, we also introduce
\[
m_n\q:=  \sup_{ t<\tau_{n+1}^q-\tau_n^q}\xi^{(n)}_t\quad \textrm{and} \quad i_n\q:= \inf_{t<\tau_{n+1}^q-\tau_n^q}\xi^{(n)}_t.
\]
\begin{lemma}\label{lemadoney}
	Using the above notation we have
	$${M}_n\q={S}_n^{(+, q)}+m_0\q, \qquad {I}_n\q={S}^{(-, q)}_n+i_0\q, \qquad n\geq 0$$
	where each the processes ${S}^{(+, q)}=({S}_n^{(+, q)},n\geq 0)$ and ${S}^{(-, q)}=({S}^{(-, q)}_n,n\geq 0)$ are random walks with the same distribution as $S\q$. Moreover ${S}^{(+, q)}$ and $m_0\q$ are independent, as are ${S}^{(-, q)}$ and $i_0\q$.
\end{lemma}
The proof of this lemma is based on the Wiener-Hopf 
factorisation (see Equations (4.3.3) and (4.3.4) in \cite{Doney}).
It follows from similar arguments as those used in the proof of Theorem IV.13 in \cite{Doney}, 
which considers the case when the exponential random variables 
are jump times of the process  $\xi$ restricted to $\R\setminus [-\eta, \eta],$ for $ \eta >0$. So, we omit its proof for the sake of brevity.

Recall that $\tau_1^{q}$ goes to $0$, in probability, as $q$ increases and that  $\xi$ has c\`adl\`ag paths. Hence, there exists an increasing sequence $(q_n)_{ n\geq 0}$ such that $q_n\rightarrow \infty$ and 
\begin{equation}\label{sequence}
e^{\lambda i_0^{(q_n)}}\underset{n\to\infty}{\longrightarrow}1, \qquad  \textrm{a.s.}
\end{equation}
We also  recall the following form of the Wiener-Hopf factorisation,  for $q>\psi(\lambda)$
\begin{align}\label{whf}
\frac{q}{q-\psi(\lambda)}=\Exp{e^{ \lambda i_0\q}}\Exp{e^{  \lambda m_0\q}}.
\end{align}
From  the {Dominated Convergence}
Theorem and identity (\ref{whf}), it follows that for  $\eps\in(0,1)$,  there exists $N\in \N$ such that for all $n\geq N$
\begin{align}\label{limitedemi}
1-\eps\leq \Exp{e^{\lambda  i_0^{(q_n)}}} \leq \Exp{e^{ \lambda m_0^{(q_n)}}}\leq 1+\eps.
\end{align}
Next, we introduce the compound Poisson process
$$Y_t\q:=S\q_{N_t\q},\qquad t\geq 0,$$
whose Laplace exponent satisfies
\begin{align*}
\psi\q(\lambda):=\log\mathbb{E}\Big[e^{\lambda Y_1\q}\Big]=\frac{q\psi(\lambda)}{q-\psi(\lambda)},
\end{align*}
which is well defined for $\lambda$ such that $q>\psi(\lambda)$.
Similarly, we define 
$$\widetilde{I}_t\q={I}\q_{N_t\q}, \qquad \widetilde{M}\q_t={M}\q_{N_t\q},  \qquad Y^{(+, q)}_t={S}^{(+, q)}_{N_t\q}, \qquad \mbox{ and } \qquad Y^{(-, q)}_t={S}^{(-, q)}_{N_t\q}.$$
We observe from the definitions of $\widetilde{M}\q$ and $\widetilde{I}\q$, and Lemma \ref{lemadoney},
that for all $t\geq 0$, the following {inequalities are} satisfied
\begin{equation}\label{cotas}
\begin{split}
e^{- m_0\q}\int_0^t e^{- Y^{(+,q)}_s}\ud s\le \int_0^t e^{-\xi_s}\ud s \le e^{- i_0\q}\int_0^t e^{- Y^{(-,q)}_s}\ud s.
\end{split}
\end{equation}
We have now all the tools needed to prove Theorem \ref{polinomial}.
\begin{proof}[Proof of Theorem \ref{polinomial}] 

i) 
Assume that $\psi'(0+)>0$. According to Theorem 1 in \cite{MR2178044}, $I_t(\xi)$ converges a.s. 
to $I_\infty(\xi)$, a non-negative and finite limit as $t$ goes to $\infty$. 	
Then, we observe that the  result follows from the Monotone Convergence Theorem.

	We now prove part ii). In order to do so, we use the approximation and notation that we introduced
	at the beginning of this section.
	Let $(q_n)_{n\ge 1}$ be a sequence defined as in (\ref{sequence}) and  observe that for  $n\ge 1$, we have  
	$\psi^{ (q_n)}(0)=0$, $\psi^{\prime (q_n)}(0+)=0$ and $\psi^{\prime\prime (q_n)}(0+)<\infty$. We also observe that the processes 
	$Y^{(+,q_n)}$ and $Y^{(-,q_n)}$ have bounded variation paths.
	
	We take $\ell\geq N$ and $0<\eps< 1$. Hence from Lemmas 13 and 14 in Bansaye et al. \cite{BPS}, we observe that there 
	exists a positive constant $c_1(\ell)$ such that
	$$(1-\eps)c_1(\ell)t^{-1/2}\leq \Exp{\left(\int_0^t e^{-Y^{(\pm, q_\ell)}_s}\ud s\right)^{-p}}\leq (1+\eps)c_1(\ell)t^{-1/2}, \quad \textrm{ as }t\to \infty.$$
	Therefore using (\ref{limitedemi}) and (\ref{cotas}) in the previous inequality, we obtain
	\begin{align}\label{desigualdadc1}
	(1-\eps)^2c_1(\ell)t^{-1/2}\leq \E [ I_{t}(\xi)^{-p} ] \leq (1+\eps)^2c_1(\ell)t^{-1/2}, \quad \textrm{ as }t\to \infty.
	\end{align}
	Next, we take $n,m\ge N$ and observe that the previous inequalities imply
	$$\left(\frac{1-\eps}{1+\eps}\right)^2c_1(n)\leq c_1(m)\leq \left(\frac{1+\eps}{1-\eps}\right)^2c_1(n),\qquad \mbox{ for all } n,m\ge N.$$
	Thus, we deduce that  $(c_1(n))_{ n\ge 1}$ is a Cauchy sequence. Let us denote $c_1$ its limit which, by the previous inequalities is positive.  Let $k\geq N$ such that
	$$(1-\eps)c_1\leq c_1(k)\leq (1+\eps)c_1.$$
	Using this inequality and  (\ref{desigualdadc1}), we observe 
	$$(1-\eps)^3c_1t^{-1/2}\leq \E [ I_{t}(\xi)^{-p} ] \leq (1+\eps)^3c_1t^{-1/2}, \quad \textrm{ as }t\to \infty.$$
	This completes the proof of part ii).
	
	Now, we prove part iii)-a).   {Recalling \eqref{timereversal} yields that}
	\begin{equation}\label{duality}
	\begin{split}
	I_t(\xi)\stackrel{(d)}{=}e^{-\xi_t}I_t(-\xi), \qquad t\geq 0.
	\end{split}
	\end{equation}
	Hence using the Esscher transform (\ref{escheerk}), with $\lambda=p$, we have
	\begin{equation}\label{eqesscher}
	\E \left[ I_t(\xi)^{-p} \right] =   \E \left[ e^{p\xi_t} I_{t}(-\xi)^{-p} \right] =   
	e^{t\psi(p)} \E^{(p)} \left[ I_{t}(-\xi)^{-p} \right], \qquad t\geq 0. 
	\end{equation}
	The inequality \eqref{mdoob} with $q=p$ and the previous identity imply that the decreasing function $t \mapsto \E^{(p)} [ I_{t}(-\xi)^{-p} ] $ is  
	finite for all $t>0$. Recall that under the probability 
	measure $\P^{(p)}$, the process $\xi$ is a L\'evy process with  mean  $\E^{(p)}[\xi_1]=\psi'(p)\in (-\infty, 0)$. Then, as in the proof of part i), 
	$ \E^{(p)} [ I_{t}(-\xi)^{-p} ] $  converges to $ \E^{(p)} [ I_{\infty}(-\xi)^{-p} ], $ as $t$ increases.
	
	Part iii)-b) follows from part ii) and the Esscher transform (\ref{escheerk}). More precisely, we  apply  the Esscher transform with 
	$\lambda=p$ and observe that the Laplace transform of the process $\xi$ under the probability measure $\P^{(p)}$, satisfies
	$\psi_{p}'(0+)=\psi'(p)=0$ and $\psi_{p}^{\prime\prime}(0+)=\psi^{\prime\prime}(p)<\infty$. Therefore  by applying part ii) and identity \eqref{eqesscher}, we {get the existence of} a constant $c_2>0$ such that
	$$\Exp{I_{t}(\xi)^{-p} }=e^{t\psi(p)}\E^{(p)} [ I_{t}(-\xi)^{-p} ] \sim c_2t^{-1/2}e^{t\phi(p)}.$$\\
	
	Finally we prove part iii)-c). Again from the Esscher transform with $\lambda=\tau$, we see
	\[
	\Exp{I_{t}(\xi)^{-p} }=e^{t\psi(\tau)}\E^{(\tau)} [e^{(p-\tau)\xi_t} I_{t}(-\xi)^{-p} ],\qquad t> 0.
	\]
	On the one hand, for $t>0$,
	\[
	\begin{split}
	\E^{(\tau)} [e^{(p-\tau)\xi_t} I_{t}(-\xi)^{-p} ]&=\E^{(\tau)} \left[e^{(p-\tau)(\xi_t-\xi_{t/2})} \frac{(e^{-\xi_{t/2}}I_{t/2}(-\xi)+\int_{t/2}^te^{\xi_u-\xi_{t/2}}{\ud u})^{-(p-\tau)}}{(I_{t/2}(-\xi)+e^{\xi_{t/2}}\int_{t/2}^te^{\xi_u-\xi_{t/2}}{\ud u})^{\tau}} \right]\\
	&\le\E^{(\tau)} \left[e^{(p-\tau)(\xi_t-\xi_{t/2})} \frac{(\int_{0}^{t/2}e^{\xi_{s+t/2}-\xi_{t/2}}{\ud s})^{-(p-\tau)}}{I_{t/2}(-\xi)^{\tau}} \right]\\
	&= \E^{(\tau)} \left[e^{(p-\tau)(\xi_{t/2})} I_{t/2}(-\xi)^{-(p-\tau)}\right]\E^{(\tau)} \left[I_{t/2}(-\xi)^{-\tau} \right],
	\end{split}
	\]
	where we have {used} in the last identity the fact that $(\xi_{u+t/2}-\xi_{t/2},  u\ge 0)$ is independent of  $(\xi_u, 0\le u\le t/2)$ and with the same law as $(\xi_u,u\ge 0)$.
	
	On the other hand, from  (\ref{duality}) we deduce
	\[
	\E^{(\tau)} \left[e^{(p-\tau)(\xi_{t/2})} I_{t/2}(-\xi)^{-(p-\tau)}\right]= \E^{(\tau)} \left[ I_{t/2}(\xi)^{-(p-\tau)}\right],\qquad t> 0.
	\]
	Putting all the pieces together, we get
	\[
	\E^{(\tau)} [e^{(p-\tau)\xi_t} I_{t}(-\xi)^{-p} ]\le \E^{(\tau)} \left[ I_{t/2}(\xi)^{-(p-\tau)}\right]\E^{(\tau)} \left[I_{t/2}(-\xi)^{-\tau} \right],\qquad t> 0
	\]
	implying 
	\[
	\Exp{I_{t}(\xi)^{-p} }\le e^{t\psi(\tau)}\E^{(\tau)} \left[ I_{t/2}(\xi)^{-(p-\tau)}\right]\E^{(\tau)} \left[I_{t/2}(-\xi)^{-\tau} \right],\qquad t> 0.
	\]
	Since $\psi'(\tau)=0$,  we have $\mathbb{E}^{(\tau)}[\xi_1]=0$ and the process $\xi$ oscillates under  $\mathbb{P}^{(\tau)}$.
	Moreover since $\psi^{\prime\prime}(\tau)<\infty$, we deduce that $\psi^{\prime\prime}_\tau(0+)<\infty$. The latter condition implies 
	from part ii) that there exists a constant $c_1(\tau)>0$ such that
	\[
	\E^{(\tau)} [ I_{t}(\xi)^{-(p-\tau)} ]\sim c_1(\tau) t^{-1/2}\qquad \textrm{as} \quad t\to \infty.
	\] 
	Since the process $\xi$ oscillates under  $\mathbb{P}^{(\tau)}$, the dual $-\xi$ also oscillates. This implies that $I_{t}(-\xi)$ goes to $\infty$ and  therefore $\E^{(\tau)} [ I_{t}(-\xi)^{-(p-\tau)} ]$ goes to 0, as $t$ increases. 
	In other words, we have
	$$\Exp{I_t(\xi)^{-p}}=o(t^{-1/2}e^{t\psi(\tau)}), \qquad \textrm{as} \quad t\rightarrow \infty,$$
	as expected.
	
	We now assume that $\xi$ is non-arithmetic, our arguments are similar to those used in \cite{BPS}. We will prove 
	$$\limsup_{t\rightarrow\infty}t^{3/2}e^{-t\psi(\tau)}\Exp{I_t(\xi)^{-p}}< \infty.$$
	In order to prove it, we take $t>0$ and observe 
	$$I_{\lfloor t\rfloor}(\xi)=\underset{k=0}{\overset{\lfloor t\rfloor-1}{\sum}}e^{-\xi_k}\int_0^1e^{-(\xi_{k+u}-\xi_k)}\ud u.$$
	Therefore
	$$\Exp{I_{\lfloor t\rfloor}(\xi)^{-p}}\leq \Exp{\underset{k\leq \lfloor t\rfloor-1}{\min}e^{p\xi_k}\left(\int_0^1e^{-(\xi_{k+u}-\xi_k)}\ud u\right)^{-p}}.$$	
	Conditioning on the value when the minimum is {attained, let} say {$k^\prime$},  and observing that $e^{p\xi_{k^\prime}}$
	is independent of $\left(\int_0^1e^{-(\xi_{k^\prime+u}-\xi_{k^\prime})}\ud u\right)^{-p}$ and the latter has the same law as $\left(\int_0^1e^{-\xi_{u}}\ud u\right)^{-p}$, we deduce
	$$\Exp{I_{\lfloor t\rfloor}(\xi)^{-p}}\leq \Exp{\underset{k\leq \lfloor t\rfloor-1}{\min}e^{p\xi_k}}\Exp{\left(\int_0^1e^{-\xi_{u}}\ud u\right)^{-p}}.$$	
	Finally, by Lemma 7 in \cite{MR1633937}, there {exists} a $C>0$ such that 
	$$\Exp{\underset{k\leq \lfloor t\rfloor -1}{\min}e^{p\xi_k}}\sim C\lfloor t\rfloor^{-3/2}e^{\lfloor t\rfloor\psi(\tau)},\qquad \textrm{for $t$ large.}$$
	The claim follows from the monotonicity of $\Exp{I_{\lfloor t\rfloor}(\xi)^{-p}}$ and the fact that $t\in (\lfloor t\rfloor, \lfloor t\rfloor +1)$.
\end{proof}
The idea of the proof of Theorem \ref{funcLev}
is to study the asymptotic behaviour of $\mathcal{E}_F(n/q)$ for $q$ fixed and large $n$, 
and then to use the monotonicity of $F$ to deduce the asymptotic behaviour of $\mathcal{E}_F(t)$ when 
$t$ goes to infinity.
In order to do so,  we  use a key result due to Guivarc'h and Liu that we state in the Appendix 
for the sake of completeness.

Let $q>0$ and define the sequence $q_n=n/q,$ for $n\ge 0$. For $k\ge 0$, we also define
\[
\widetilde{\xi}^{(k)}_u=\xi_{q_k+u}-\xi_{q_k}, \qquad \textrm{for }\quad u\ge 0,
\]
and 
\begin{equation}\label{sequences}
a_k=e^{-\widetilde{\xi}^{(k)}_{q_{(k+1)}-q_k}}\qquad\textrm{and}\qquad b_k=\int_0^{q_{(k+1)}-q_k}e^{-\widetilde{\xi}^{(k)}_{u}}\ud u.
\end{equation}
Hence, $(a_k,b_k)$ is a $\R_+^2$-valued sequence of {i.i.d.} random variables. Observe that
\[
a_0=e^{-\xi_{\frac{1}{q}}} \qquad \textrm{and}\qquad \frac{b_0}{1-a_0}=\frac{I_{\frac{1}{q}}(\xi)}{1-e^{-\xi_{\frac{1}{q}}}},
\]
which are not constant a.s. as required by Theorem \ref{th3}. Moreover, we have
\[ 
\int_{q_{i}}^{q_{i+1}}e^{- \xi_u}\ud u = 
e^{- \xi_{q_{i}}} b_{i}=\prod_{k=0}^{i-1} a_k b_{i}=A_{i}b_{i},
\]
where $A_k$ is defined as in Theorem \ref{th3}. The latter identity implies
$$ I_{q_n}(\xi)=\sum_{i=0}^{n-1} \int_{q_{i}}^{q_{i+1}}e^{- \xi_u}\ud u  = 
\sum_{i=0}^{n-1} A_i b_{i}:=B_{n}.
$$
In other words, we have all the objects required  to apply Theorem \ref{th3}.

\begin{proof}[Proof of Theorem \ref{funcLev}]
	
	i) The proof uses similar arguments as those used in the proof of Theorem \ref{polinomial}-i).

	{ii)} We now assume that $\psi'(0+)=0$. We define the sequence $(a_k, b_k)_{k\ge 0}$  as in (\ref{sequences}) and follow the 
	same notation as in Theorem \ref{th3}. We take $0<\eta<\alpha$ and $d_p>1$ such that 
	$-\theta^-/d_p<p$ and $ \theta^-<-\eta<\eta+p<\theta^+$, and let
	$$ (\eta, \kappa, \vartheta)= \left(\eta, \frac{-\theta^-}{d_p},p\right).$$
	Next, we verify  the {moment conditions} of Theorem \ref{th3} for  the couple $(a_0, b_0)$. From the definition of $(a_0, b_0)$, it is clear
	\[
	\mathbb{E}\left[\ln a_0\right]=\frac{\psi^\prime(0+)}{q}=0, \qquad \mathbb{E}\left[a_0^{\kappa}\right]=e^{\psi(-\kappa)/q}\quad\textrm{ and }\quad \mathbb{E}\left[a_0^{-\eta}\right]=e^{\psi(\eta)/q},
	\]
	which are well defined.
	Similarly as in (\ref{mdoob}), by $L_1$-Doob's 
	inequality (see \cite{acciaio2012trajectorial}) and the Esscher transform \eqref{escheerk}
	\[
	\begin{split}
	\mathbb{E}\left[b_0^\eta\right]\le q^{\eta}\mathbb{E}\left[\sup_{0\le u\le 1/q} e^{-\eta \xi_u}\right]
	\le \frac{e}{e-1}q^{\eta}e^{\frac{\psi(-\eta)}{q}}\left(1-\eta\psi'(-\eta)-\psi(-\eta)\right)<\infty,
	\end{split}
	\]
	and
\[
	\mathbb{E}\left[a_0^{-\eta}b_0^{-\vartheta}\right]
	\le q^{\vartheta}\mathbb{E}\left[e^{\eta \xi_{\frac{1}{q}}}\sup_{0\le u\le 1/q} e^{\vartheta \xi_u}\right]
	\le  q^{\vartheta}\mathbb{E}\left[\sup_{0\le u\le 1/q} e^{(\eta+\vartheta) \xi_u}\right]<\infty.
	\]
Therefore the asymptotic behaviour of $\mathcal{E}_F(q_n)$ for large $n$, follows from  a direct application of Theorem \ref{th3}. In other words, there exists a 
	positive constant $c(q)$ such that 
	\begin{equation*}
	\sqrt{n}\mathcal{E}_F(q_n) \sim c(q), \qquad\textrm{ as }\quad n \to \infty.  
	\end{equation*}
	In order to get our result, we take $t$ to be a positive real number. Since  the mapping $s \mapsto \mathcal{E}_F(s)$ is non-increasing, we get
	$$
	\sqrt{t} \mathcal{E}_F(t)\leq \sqrt{t}\mathcal{E}_F(\lfloor qt \rfloor /q)
	=\sqrt{\frac{t}{\lfloor qt \rfloor}}\sqrt{\lfloor qt \rfloor}\mathcal{E}_F(\lfloor qt \rfloor /q).
	$$
	Similarly
	$$
	\sqrt{t} \mathcal{E}_F(t) \geq \sqrt{t}\mathcal{E}_F((\lfloor qt \rfloor +1)/q)
	=\sqrt{\frac{t}{\lfloor qt \rfloor+1}}\sqrt{\lfloor qt \rfloor+1}{\mathcal{E}_F}((\lfloor qt \rfloor +1)/q).
	$$
	Therefore
	\[
	\sqrt{t} \mathcal{E}_F(t) {\sim}c(q)q^{-1/2}, \qquad\textrm{ as } t \to \infty.
	\]
	Moreover, we deduce that
	$  c(q)q^{-1/2}$ is positive 
	{and does not depend on $q$. Hence we denote this constant by $c_1$.} This concludes the proof of point {ii)}.\\

	{iii)} For the rest of the proof, we assume that $\psi'(0)<0$. We first prove part {a)}.  Since $\psi^\prime(p)<0$,  from Theorem \ref{polinomial} part {iii)}-a) 
	we {know} that 
	$$\Exp{I_t(\xi)^{-p}}\sim  e^{t\psi(p)}\mathbb{E}^{(p)}[I_\infty(-\xi)^{-p}], \qquad as \quad t\rightarrow \infty.$$
	Hence the asymptotic behaviour is proven if we show that
	\begin{equation*} \mathcal{E}_F(t) \sim \mathbf{k} \E\left[  I_t(\xi)^{-p}\right], \quad \textrm{as }\quad t\to\infty.\end{equation*}
	Since $\psi'(p)<0$, there is $ \varepsilon>0 $ such that  $p(1+\varepsilon)<{\theta^+}$, ${\psi}(p(1+\varepsilon) )<\psi(p)$ and 
	$\psi^{\prime}((1+\varepsilon)p)<0$. Hence, from Lemma 2 {(see the Appendix)}, we deduce that there is a constant $M$ such that
	\begin{equation}\label{sim1}
	\left| F\left( I_t(\xi)\right)-\mathbf{k}I_t(\xi)^{-p} \right|
	\leq M I_t(\xi)^{-(1+\varepsilon)p}. \end{equation}
	In other words, it is enough  to prove
	\begin{equation} \label{limit1}
	\E \left[ I_{t}(\xi)^{-(1+\varepsilon)p} \right]=o(e^{t{\psi}(p)}), 
	\quad \textrm{as }\quad t\to\infty. 
	\end{equation} 
{From the Esscher transform (\ref{escheerk}) with $\lambda=(1+\varepsilon)p$, we deduce
	\[
	\E \left[ I_t(\xi)^{-(1+\varepsilon)p} \right]  = 
	\E \left[  e^{p(1+\varepsilon)\xi_s}I_t(-\xi)^{-(1+\varepsilon)p} \right] 
	=   e^{t\psi(p)}e^{t{\psi}_p(\varepsilon p)} \E^{((1+\varepsilon)p)} \left[I_t(-\xi)^{-(1+\varepsilon)p}
	\right] .
	\]
	This and Equation \eqref{mdoob} with $\lambda=(1+\varepsilon)p$ imply that  $ \E^{((1+\varepsilon)p)} [ I_{t}(-\xi)^{-(1+\varepsilon)p} ] $ is  
	finite for all $t>0$.} Similarly as in the proof of Theorem \ref{polinomial} {iii)}-a), we can deduce that  
	$\E^{((1+\varepsilon)p)} [ I_{t}(-\xi)^{-(1+\varepsilon)p} ] $ has a finite limit, as $t$ goes to $\infty$. 
	We conclude by observing that ${\psi}_p(\varepsilon p)$ is negative implying that (\ref{limit1}) holds. We complete the proof of point iii)-a)
	by observing that \eqref{sim1} and \eqref{limit1} {yield}
		$$ \E[F\left( I_t(\xi)\right)]\sim \mathbf{k}\E[I_t(\xi)^{-p}], \quad t \to \infty. $$\\

	We now prove part {b)}.  Since $\psi^\prime(p)=0$ and $\psi^{\prime\prime}(p)<\infty$,  from Theorem \ref{polinomial} part {iii)}-{b)}
	we {know} that there exists a positive constant $c_2$ such that
	$$\Exp{I_t(\xi)^{-p}}\sim c_2 t^{-1/2}e^{t\psi(p)}, \qquad as \quad t\rightarrow \infty.$$
	Similarly as in the proof of part {a)}, the asymptotic behaviour is proven if we show that
	\[
	\mathcal{E}_F(t)  \sim {\mathbf{k}}  \E\left[  I_t(\xi)^{-p}\right], \quad \textrm{as }\quad t\to\infty,
	\]
	which amounts to showing that  
	\begin{equation*} 
	\E \left[ I_{t}(\xi)^{-(1+\varepsilon)p} \right]=o(t^{-1/2}e^{t{\psi}(p)}), 
	\quad \textrm{as }\quad t\to\infty
	\end{equation*} 
	for $\eps$ small enough.
	The latter follows from of Theorem \ref{polinomial} {iii)}-{c)}.\\

	Finally, we prove part {c)}.  Similarly as in the proof of  part {ii)}, we define the sequence 
	$(a_k, b_k)_{k\ge 0}$  as in (\ref{sequences}) and follow the same notation as in Theorem \ref{th3}.
	Let us choose $0<\eta<\alpha$ such that $0<\tau-\eta < \tau+p+\eta<\theta^+$ and take
	$$ (\eta, \kappa, \vartheta)= \left(\eta, \tau,p\right).$$
	
	Next, we apply the Esscher transform \eqref{escheerk} with $\lambda = \tau$ and observe 
	\begin{equation}\label{chgtproba}
	\E[F(I(q_n))]e^{-q_n{\psi}(\tau)}=\E^{(\tau)}[e^{-\tau {\xi}_{q_n}}F(I(q_n))]= \E^{(\tau)}[A_n^\tau F(B_n)] .
	\end{equation}
	Hence in order to apply Theorem \ref{th3},  we need the {moment conditions on} $(a_0, b_0)$ to be 
	satisfied under the probability measure $\p^{(\tau)}$. We first observe, 
	\[
	\E^{(\tau)}[\ln a_0]=\E^{(\tau)}[\xi_{1/q}]=e^{-\psi(\tau)/q}\e[\xi_{1/q}e^{\tau \xi_{1/q}}]=\frac{\psi^\prime(\tau)}{q}=0.
	\]
	Similarly, we get
	\[
	\mathbb{E}^{(\tau)}\left[a_0^{\kappa}\right]=\e^{(\tau)}[e^{-\kappa \xi_{1/q}}]=e^{-\psi(\tau)/q}\quad\textrm{ and }\quad 
	\mathbb{E}^{(\tau)}\left[a_0^{-\eta}\right]=\e^{(\tau)}[e^{\eta \xi_{1/q}}]=e^{\psi_\tau(\eta)/q},
	\]
	where $\psi_{\tau}(\lambda)=\psi(\tau+\lambda)- \psi(\tau)$. From our assumptions both expectations are finite.
	
	Again, we use similar arguments as those used in  (\ref{mdoob}) to deduce
	$$
	\mathbb{E}^{(\tau)}\left[b_0^\eta\right]
	\le q^{-\eta}\mathbb{E}^{(\tau)}\left[\sup_{0\le u\le 1/q} e^{-\eta \xi_u}\right]\le 
	q^{-\eta}e^{-\psi(\tau)/q}\mathbb{E}\left[\sup_{0\le u\le 1} e^{(\tau-\eta) \xi_u}\right]
	<\infty,
	$$
	and
	$$
	\mathbb{E}^{(\tau)}\left[a_0^{-\eta}b_0^{-p}\right]\le q^{p}\mathbb{E}^{(\tau)}\left[e^{\eta \xi_{\frac{1}{q}}}\sup_{0\le u\le 1/q} 
	e^{p \xi_u}\right]\le
	q^{p}e^{-\psi(\tau)/q}\mathbb{E}\left[\sup_{0\le u\le 1} e^{(\tau+\eta+p) \xi_u}\right]
	<\infty.
	$$
	Therefore the asymptotic behaviour of $\E^{(\tau)}[A_n^\tau F(B_n)]$ follows from a direct application of Theorem \ref{th3} 
	with the functions $\tilde{\psi}(x)=F(x)$ and $\tilde{\phi}(x)= x^{\tau} $. 
	In other words, we conclude that there exists a 
	positive constant $c(q)$ such that 
	\begin{equation*}
	n^{3/2}\E^{(\tau)}[A_n^{\tau}F(B_n)] \sim c(q), \quad n \to \infty.
	\end{equation*}
	In particular from \eqref{chgtproba}, we deduce 
	$$ {\mathcal{E}}_F(q_n) \sim  c(q) e^{-n{\psi}(\tau)/q}n^{-3/2} , \quad n \to \infty .$$
	Then using the monotonicity of $F$ as in the proof of part ii), we get that for $n$ large enough,
	\begin{equation}\label{eq}
	c(q)q^{-3/2} e^{-{\psi}(\tau)/q}  
	\leq n^{3/2}e^{n{\psi}(\tau)}\mathcal{E}_F(n)\leq 
	c(q)q^{-3/2} .
	\end{equation}
	A direct application of Lemma \ref{lemtec} then yields the existence of a nonnegative constant $c_4$ such that
	$$\lim_{q \to \infty}c(q)q^{-3/2}=c_4.  $$
	Moreover, \eqref{eq} yields that $c_4$ is positive.
	This ends the proof.
\end{proof}

\section{Appendix}

We recall in this section a Theorem due to Guivarc'h and Liu (see Theorem 2.1 in \cite{GL}) and two technical Lemmas stated in \cite{BPS}:

\begin{theorem}[Guivarc'h, Liu 01]
	\label{th3} 
	Let $(a_n,b_n)_{n \geq 0}$ be a $\R_+^2$-valued sequence of {i.i.d.} random variables such that $\E[\ln a_0]=0$. 
	Assume that $b_0/(1-a_0)$ is not constant a.s. and define
	\[
	A_0:=1, \quad A_n:=\prod_{k=0}^{n-1}a_k\quad \textrm{ and }\quad B_n:=\sum_{k=0}^{n-1}A_kb_k, \qquad \textrm{ for }\quad n\geq 1.
	\] 
	Let $ \eta, \kappa, \vartheta$ be three positive numbers such that $\kappa <\vartheta$, and $\tilde{\phi}$ and $\tilde{\psi}$ be 
	two positive continuous functions on $\R_+$  such that they do not vanish and for a constant $C>0$ and for every $a>0$, $b\geq 0$, 
	$b'\geq 0$, we have
	$$ \tilde{\phi}(a)\leq C a^\kappa, \quad \tilde{\psi}(b)\leq \frac{C}{(1+b)^\vartheta}, \quad\text{ and }\quad 
	|\tilde{\psi}(b)-\tilde{\psi}(b')|\leq C |b-b'|^\eta. $$
	Moreover, assume that 
	$$ \E\big[a_0^\kappa\big]<\infty,\quad \E\big[a_0^{-\eta}\big]<\infty,\quad \E\big[b_0^{\eta}\big]<\infty\quad\text{ and }
	\quad \E\big[a_0^{-\eta}b_0^{-\vartheta}\big]<\infty. $$
	Then, there exist two positive constants $c(\tilde{\phi},\tilde{\psi})$ and $c(\tilde{\psi})$ such that 
	$$ \underset{n \to \infty}{\lim}n^{3/2}\E \left[ \tilde{\phi}(A_n)\tilde{\psi}(B_n) \right]=c(\tilde{\phi},\tilde{\psi})
	\qquad \text{ and }\qquad \underset{n \to \infty}{\lim}n^{1/2}\E \left[ \tilde{\psi}(B_n) \right]=c(\tilde{\psi}).  $$
\end{theorem}

\begin{lemma} \label{lemtecaussi}
	Assume that $F$ satisfies one of the Assumptions ({\bf A1}) or ({\bf A2}). 
	Then there exist two positive  finite constants $\eta$ and $M$ such that for all  $(x,y)$ in $\R_+^2$ and $\varepsilon $ in $[0,\eta]$,
	\begin{eqnarray*}
	\Big| F(x)-\mathbf{k} x^{-p} \Big| &\leq & Mx^{-(1+\varepsilon)p}, \\
	\Big| F(x)-F(y) \Big| &\leq &M\Big| x^{-p}-y^{-p} \Big|.
	\end{eqnarray*}
\end{lemma}

\begin{lemma}\label{lemtec}
	Assume that the non-negative sequences $(a_{n,q})_{(n,q) \in \N^2}$, $(a'_{n,q})_{(n,q) \in \N^2}$ and $(b_{n})_{n \in \N}$ satisfy for every $(n,q) \in \N^2$:
	$$ a_{n,q}\leq b_n \leq a'_{n,q}, $$
	and that there exist  three sequences $(a(q))_{q \in \N}$, $(c^-(q))_{q \in \N}$ and $(c^+(q)_{q \in \N}$ such that
	$$ \underset{n \to \infty}{\lim}a_{n,q}=c^-(q)a(q), \quad \underset{n \to \infty}{\lim}a'_{n,q}=c^+(q)a(q), \quad \text{and} \quad \underset{q \to \infty}{\lim}c^-(q)=\underset{q \to \infty}{\lim}c^+(q)=1. $$
	Then there exists a non-negative constant $a$ such that 
	$$ \underset{q \to \infty}{\lim}a(q)=\underset{n \to \infty}{\lim}b_{n}=a .$$
\end{lemma}

\bibliographystyle{abbrv}
\bibliography{bibliofuncexp}

\end{document}